\documentclass[leqno,10pt]{amsart} 
\usepackage{amsmath,amssymb} 
\usepackage{amscd}

\begin{document}
%
%

\theoremstyle{plain}
\newtheorem{thm}{\indent Theorem}[section] 
\newtheorem{lem}{\indent Lemma} 
\newtheorem{cor}{\indent Corollary} 
\newtheorem{prop}{\indent Proposition} 
\newtheorem{claim}{\indent Claim} 
\theoremstyle{remark}
\newtheorem{define}{\indent Definition} 
\newtheorem{rem}{\indent Remark} 
\newtheorem{example}{\indent Example} 
\newtheorem{notation}{\indent Notation} 
\newtheorem{assertion}{\indent Assertion} 
\newtheorem{conj}{\indent Conjecture} 
\newtheorem{pf}{\indent Proof}
\renewcommand{\thepf}{}
\newtheorem{formula}{\indent Formula}
\newtheorem{ass}{\indent Assumption}

\newtheorem{remarkNo}{\indent Remark.}
\renewcommand{\theremarkNo}{}

\def\C{{\mathbb{C}}}
\def\R{{\mathbb{R}}}
\def\Z{{\mathbb{Z}}}
\def\Q{{\mathbb{Q}}}
\def\F{{\mathbb{F}}}
\def\g{{\mathfrak{g}}}
\def\gl{{\mathfrak{gl}}}
\def\h{{\mathfrak{h}}}
\def\sl{{\mathfrak{sl}}}
\def\sp{{\mathfrak{sp}}}
\def\D{{\mathfrak{D}}}
\def\p{{\mathfrak{p}}}
\def\o{{\mathcal{O}_K}}
\def\Tr{\mathrm{Tr}}
\def\N{\mathrm{N}}
\def\K{\hat{K}}
\def\hatK{\hat{\hat{K}}}

\makeatletter
\def\qedsymbol{\hbox to.77778em{\hfil\vrule
\vbox to.675em{\hrule width.6em\vfil\hrule}\vrule\hfil}}
\def\qed{{\unskip\nobreak\hfil\penalty50
\quad\hbox{}\nobreak\hfil \hbox{$\qedsymbol$}
\parfillskip\z@ \finalhyphendemerits\z@\par}}
\long\def\proof{%
\trivlist\itemindent0pt
\item[\hskip \labelsep{\indent \scshape Proof.}]\hskip.5em}
\def\endproof{\hfill \qed \endtrivlist}
\def\address#1#2{\begingroup
\noindent\parbox[t]{7.5cm}{%
\small{\scshape\ignorespaces#1}\par\vskip1ex
\noindent\small{\itshape E-mail address}%
\/: #2\par\vskip4ex}\hfill%
\endgroup}%
\makeatother
\title{On ramifications of Artin-Schreier extensions of surfaces over algebraically closed fields of positive characteristic III}
\author{Masao Oi} 
\thanks{The author is supported by the JSPS Research Fellowships for Young Scientists.}
\begin{abstract}
      For a smooth surface $X$ over an algebraically closed field of positive characteristic,
      we consider the ramification of an Artin-Schreier extension of $X$.
      A ramification at a point of codimension 1 of $X$ is understood by the Swan conductor.
      A ramification at a closed point of $X$ is understood by the invariant $r_x$ defined by Kato [\ref{ch:Ka2}].
      The main theme of this paper is to give a simple formula to compute $r_x'$ defined in [\ref{ch:Oi2}],
      which is equal to $r_x$ for good Artin-Schreier extension.
      We also prove Kato's conjecture for upper bound of $r_x$.
\end{abstract}
%

\maketitle

\section{Introduction}
     Let $F$ be an algebraically closed field of characteristic $p>0$.
     Let $X$ be a smooth proper surface over $F$, $D$ a simple normal crossing divisor. Put $U = X - D$.
     For a closed point $x \in D$, we say $x$ is in Case (I) if the number of irreducible components of $D$ containing $x$ is one,
     we say $x$ is in Case (II) if the number of irreducible components of $D$ containing $x$ is two.
     We also say $x$ is of type (I) or type (II).
     In this paper, we always assume $D$ is generated by $t_1$ in $O_{X,x}$ in case (I), and
     $D$ is generated by $t_1 t_2$ in $O_{X,x}$ in case (II), where $t_1,t_2 \in O_{X,x}$.
     Throughout this paper, we fix $x$ once and for all. 
     
     Let $l$ be a prime number which is different from $p$.
     For a character $\chi : \pi_1(U) \to \bar \Q_l^\times$ of order $p$, Kato has defined an invariant $r_x = r_x(\chi)$ in his paper [\ref{ch:Ka2}].
     The invariant $r_x$ is related to the Euler Poincar\'e characteristic of $\mathfrak{F}_\chi$,
     where $\mathfrak{F}_\chi$ is the \'etale sheaf corresponding to $\chi$.
     Let $K$ be the function field of $X$, and $K'$ be the Artin-Schreier extension of $K$ corresponding to $\chi$.

     By the Artin-Schreier theory, there is an element $f \in K$ such that $K'=K(\alpha)$ and $\alpha^p-\alpha=f$.
     $f$ is determined modulo $\beta(K)$ (up to constant multiple) by $K'/K$, where $\beta$ is the Artin-Schreier map $x\mapsto x^p-x$.

     In part I of this paper ([\ref{ch:Oi1}]), we studied certain Artin-Schreier extensions of 2-dimensional affine plane over $F$ and
     found an algorithm to compute $r_x'$ (see Definition \ref{ch:r_x'}) which is equal to $r_x$ for ``almost all" extensions.
     In part II of this paper, we generalize this result to any Artin-Schreier extension of surfaces over $F$ and
     prove Kato's conjecture for good Artin-Schreier extension.
     In this paper, we give a simple formula to compute $r_x'$ (Theorem \ref{ch:simple})
     and prove Kato's conjecture for upper bound of $r_x$ (Theorem \ref{ch:Katoconj1}).

\section{Definitions of $Sw_{D'}(f)$ and $r_x$}
     We recall the definition of the Swan conductor $Sw_{D'}(f)$ following [\ref{ch:Ka1}].
\begin{equation*}
     Sw_{D'}(f) := \min \{  \max \{ -v_{D'}(g) , 0 \} \mid g \in K, g \equiv f \mod{ \beta (K) } \}.
\end{equation*}
     Here $D'$ is an irreducible divisor of $X$ and $v_{D'}$ is the normalized additive valuation on $K$ defined by $D'$. 
     Let $X'=X_s \to X_{s-1} \to \cdots \to X_0 = X$ be a sequence of blowing-ups of closed points lying over $x$
     such that $(X',U',\chi)$ is clean (see Definition \ref{ch:clean}) at all points of $X' \backslash U'$ with $U'$ the inverse image of $U$ in $X'$.
     For each $0 \leq i < s$, let $\mu_i$ be the following nonnegative integer.
     Let $U_i$ be the inverse image of $U$ in $X_i$. Let
     $\mu_i = e_i(e_i-1)$ in Case (I) (resp. $\mu_i = e_i^2$ in Case (II)) with $e_i \geq 0$
     the integer defined with respect to the blowing up $pr_i: X_{i+1} \to X_i$ at $x_i$.
     Here
\begin{align*}
     e_i  :=
\begin{cases} 
Sw_{D_{1,i}}(f) - Sw_{pr_i^{-1}(x_i)}(f) &\mbox{ \quad if $x_i$ is of type (I),} \\
Sw_{D_{1,i}}(f) + Sw_{D_{2,i}}(f) - Sw_{pr_i^{-1}(x_i)}(f) &\mbox{ \quad if $x_i$ is of type (II) ,}
\end{cases}
\end{align*} 
     where $D_{1,i}$ (resp. $D_{1,i},D_{2,i}$) is the irreducible divisors of $X_i$ containing $x_i$. 
     The invariant $r_x$ is defined by
\begin{equation*}
     r_x = \displaystyle \sum_{i=0}^{s} \mu_i.
\end{equation*}
     In Case (I), we choose $t_2$ such that $(t_1,t_2)$ is the maximal ideal of $O_{X,x}$.
     We fix $t_2$ once and for all if $x$ is of type (I).

\section{Definitions of $pg(f)$, $r'_x$, good representatives, clean models}
     Let the notations be as in the introduction.
     We recall the definitions of $pg(f)$ and $\bar{pg}(f)$ following [\ref{ch:Oi2}].
     
     Let $f$ be an element of $O_{X,x}[(t_1t_2)^{-1}]$, then
     there exist integers $k,a_0,b_0,a_1,b_1,\ldots,a_k,b_k$ such that
     $a_0<a_1<\cdots<a_k,b_0<b_1<\cdots<b_k$, and
\begin{equation}
     f \in t_1^{-a_0} t_2^{b_0} O_{X,x}^\times + t_1^{-a_1} t_2^{b_1} O_{X,x}^\times + \cdots  + t_1^{-a_k} t_2^{b_k} O_{X,x}^\times.
\end{equation}
     We define $pg(f)$ and $\bar{pg}(f)$ by
\begin{equation}
     pg(f) = ((a_0,b_0),(a_1,b_1),\cdots, (a_k,b_k)) \in (\Z^2)^{\oplus k+1},
\end{equation}
\begin{equation}
    \bar{pg}(f) = \{ (a_0,b_0),(a_1,b_1),\cdots, (a_k,b_k) \}  \subset \Z^2.
\end{equation}
     We recall the definition of a good representative of $\chi$ following [\ref{ch:Oi2}].
     We say $f$ is a good representative of $\chi$ if

\begin{equation}
     \bar{pg}(f) \cap p\Z^2 \subset \{ (0,0) \}, \bar{pg}(f) \cap (\Z_{\geq 0} \times \Z_{\leq 0}) \neq \emptyset.
\end{equation}

     We define inductively the set $ess(A)$ of essential vertices of $A=((a_0,b_0),(a_1,b_1),\cdots, (a_k,b_k))$.
     We put $i_0=0$. We define $i_l$ by induction on $l$ as follows.
\begin{equation}
     i_{l+1} := \max \{ j \mid \frac{b_j-b_{i_l}}{a_j-a_{i_l}} = \min_{i_l < j' \leq k} \{  \frac{b_{j'}-b_{i_l}}{a_{j'}-a_{i_l}} \} \} .
\end{equation}
     We define $s$ so that $i_s=k$.
     We define $ess(A)$ by $ess(A) = ((a_{i_0},b_{i_0}),(a_{i_1},b_{i_1}),\cdots, (a_{i_s},b_{i_s}))$.
\begin{define}\label{ch:clean}
    In Case (I), we say $(X,U,\chi)$ is clean at $x \in D$ if $pg(f) = ((a_0,b_0))$ or $((a_0,b_0),(a_1,b_0+1))$ for some good representative $f$.
    In Case (II), we say $(X,U,\chi)$ is clean at $x \in D$ if $pg(f) = ((a_0,b_0))$ for some good representative $f$.
\end{define}
   
      We will define $r_t'$ ($t=1$ or $t=2$ according as $x$ is of type (I) or type (II)) by
\begin{align*}
      r_t'( (a_i',b_i')_{0 \leq i \leq k} ) :&= \mu + r_2'( (a_j',b_j'-a_j')_{j \in J_a'}) + r_t'( (a_j'-b_j',b_j')_{j \in J_b'}), \\
      r_t'( (a_0,b_0) ):&= 0,
\end{align*}
      where
\begin{align*}
   J_a' := \{ j \mid b_j'-a_j' < \max_{j+1 \leq i \leq k} \{ b_i'-a_i' \} \}, \\
   J_b' := \{ j \mid a_j'-b_j' > \max_{0 \leq i \leq j-1} \{ a_i'-b_j' \} \}.
\end{align*}
      Recall that $\mu = e(e-1)$ in Case (I) (resp. $\mu = e^2$ in Case (II)) with
      $e:= \max \{ n+a_k, 0 \} + \max \{ -m, 0 \} - \max \{ \{ n-m + a_i-b_i \}_{0 \leq i \leq k} , 0 \}\geq 0$.
      Note that $e:= a_k - \min_{0 \leq i \leq k} \{ a_i-b_i \}$ if $\bar{pg}(f) \cap (\Z_{\leq 0} \times \Z_{\geq 0}) \neq \emptyset$.
      We also use the notation $r_x'((a_i',b_i')_{0 \leq i \leq k})$ instead of $r_t'((a_i',b_i')_{0 \leq i \leq k})$.

\begin{define}\label{ch:r_x'}
      We define $r_x'$ by $r_x'=r_t'(pg(f))$, where $f$ is a good representative of $\chi$.
      Note that $r_x'$ is well defined because $r_t'(pg(f))$ does not depend on the choice of a good
      representative (cf. [\ref{ch:Oi2}]).
\end{define}
      
\section{A simple formula for $r_x'$}
      In this section, we give a simple formula to compute $r_x'$  (see Definition \ref{ch:r_x'}).
\begin{lem}\label{ch:ess}
      $r_x'$ depends only on the essential vertices of $pg(f)$.
\end{lem}
\begin{proof}
      Let $pg(f) = ((a_0,b_0),(a_1,b_1),\cdots, (a_k,b_k))$.
      By the definition of $r_x'$, we see
\begin{equation}
       r_x'((a_i,b_i)_{0 \leq i \leq k}) =\mu +
                                              r_x'( ( a_j -b_j , b_j )_{j \in  J_a'} ))+
                                              r_x'( ( a_j , b_j - a_j  )_{j \in  J_b'} ))
\end{equation}
      It suffices to show the right hand side depends only on essential vertex of $pg(f)$.
      We claim that $\mu$ depends only on $ess(pg(f))$. 
      In fact $\min_{0 \leq i \leq k} \{ a_i-b_i \} = \min_{0 \leq t \leq s} \{ a_{i_t}-b_{i_t} \}$.
      Indeed if $\min_{0 \leq i \leq k} \{ a_i-b_i \} < \min_{0 \leq t \leq s} \{ a_{i_t}-b_{i_t} \}$,
      there exist $u$ and $t_1$ such that $a_u-b_u$ is minimal (i.e. $a_u-b_u = \min_{0 \leq i \leq k} \{ a_i-b_i \}$)
      and $i_{t_1} < u < i_{t_1+1}$.
      By the minimality of $a_u-b_u$,
\begin{equation}
      \frac{b_u-b_{i_{t_1}}}{a_u-a_{i_{t_1}}} < 1 < \frac{b_{i_{t_1+1}}-b_u}{a_{i_{t_1+1}}-a_u}. 
\end{equation}
      Therefore
\begin{equation}
      \frac{b_u-b_{i_{t_1}}}{a_u-a_{i_{t_1}}} < \frac{b_{i_{t_1+1}}-b_{i_{t_1}}}{a_{i_{t_1+1}}-a_{i_{t_1}}}  < \frac{b_{i_{t_1+1}}-b_u}{a_{i_{t_1+1}}-a_u}.
\end{equation}
       (we use the inequality $a/b>(a+c)/(b+d)>c/d$ for $a/b>c/d>0$, $a,b,c,d>0$).
       This contradicts the definition of essential vertices.
       Therefore $\min_{0 \leq i \leq k} \{ a_i-b_i \} = \min_{0 \leq t \leq s} \{ a_{i_t}-b_{i_t} \}$.
       It is easy to see that there exist $s_1$, $s_2$ such that
       $( a_{i_t}-b_{i_t} , b_{i_t} )_{0 \leq t \leq  s_1} = ess( ( a_j -b_j , b_j )_{j \in J_a'} )$ and
       $( a_{i_t} , b_{i_t}-a_{i_t} )_{s_2 \leq t \leq  s} = ess( ( a_j -b_j , b_j )_{j \in J_b'} )$.
       Therefore, we prove the lemma by induction on $k$ and the depth
       (recall that the depth is defined by $a_k-a_0+b_k-b_0$).
\end{proof}

\begin{define}
      For $A=((a_0,b_0),\ldots,(a_k,b_k))$,
      we define the area of $A$ by
\begin{equation}
      Area(A) := \displaystyle \sum_{t_1=0}^{k-1} (a_{t_1+1} - a_{t_1} ) \cdot (b_{t_1+1} + b_{t_1} - 2 b_0).
\end{equation}
\begin{thm}\label{ch:simple}
      We have the equality $r_x' = Area(ess(pg(f))) + (t-2) (a_k-a_0)$ for a good representative $f$,
      where $t=1$ or $t=2$ according as $x$ is of type (I) or of type(II).
\end{thm}
\begin{proof}
      By the above Lemma \ref{ch:ess}, we may assume $pg(f)=ess(pg(f))=(a_i, b_i)_{0 \leq t \leq k}$.
      Let $u$ be an integer such that $a_u-b_u$ is minimal (i.e. $a_u-b_u = \min_{0 \leq i \leq k} \{ a_i-b_i \}$).
      We treat the Case (II). The proof for the case (I) is similar and omit this.
      There exist integers $s_1$ and $s_2$ such that
\begin{equation}
      ess((a_i - b_i , b_i)_{i \in J_a'}) =   (a_i - b_i ,b_i)_{0 \leq i \leq s_1},
\end{equation}
\begin{equation}
      ess((a_i , b_i - a_i)_{i \in J_b'}) =   (a_i  ,b_i - a_i)_{s_2 \leq i \leq k},
\end{equation}
\begin{align*}
      Area_{ ess((a_i , b_i)_{0 \leq i \leq k})} 
      &= 2 \times S((a_0,b_0)- \cdots- (a_u,b_u)- (a_k,b_0)) \\
      & \quad + 2 \times S((a_u,b_u) - \cdots - (a_k,b_k) - (a_k,b_0)) \\
      &= (a_k-a_u+b_u)^2 + \\
      & \quad  2 \times S((a_0,b_0)- \cdots- (a_u,b_u)- (a_u-b_u ,b_0)) \\
      & \quad + 2 \times S((a_u,b_u) - \cdots - (a_k,b_k) - (a_k,b_0+a_k-a_u+b_u)) \\
      &= (a_k-a_u+b_u)^2 + Area_{ ess((a_i-b_i , b_i)_{i \in J_a'})} + Area_{ ess((a_i , b_i-a_i)_{i \in J_b'})} 
\end{align*}
      Here we denote the area of the $n$-gon with vertices $v_1,\cdots,v_n$ by $S(v_1-v_2-\cdots-v_n)$.
      The equality $r_x' = Area(ess(pg(f)))$ holds for $k=0$.
      We prove $r_x' = Area(ess(pg(f)))$ by induction on $k$ and the depth (recall that the depth is defined by $a_k-a_0+b_k-b_0$). In fact
\begin{align*}
      r_x'((a_i,b_i)_{0 \leq i \leq k})  &=  \mu + r_x'( (a_i-b_i , b_i)_{0 \leq i \leq u} ) + r_x'((a_i , b_i-a_i)_{u \leq i \leq k} ) \\
                                             &=  (a_k-a_u+b_u)^2 + Area_{ess((a_i-b_i , b_i)_{i \in J_a'}))} + Area_{ess((a_i , b_i-a_i)_{i \in J_b'})} \\
                                            &=  Area_{ ess((a_i , b_i)_{0 \leq i \leq k}) }.
\end{align*}
     We use the induction assumption to show the second equality. 
\end{proof}
\end{define}

\section{ A proof of Kato's conjecture for the upper bound of $r_x$.}
     We prove the Kato's conjecture (cf. [\ref{ch:Oi2}]).
\begin{thm}\label{ch:Katoconj1}
      Let $f$ be a good representative of $\chi$.
      Let $pg(f) = ((n,m),(n+a_1,m+b_1),\cdots, (n+a_k,m+b_k))$.
      The following inequality holds.
\begin{align*}
      r_x \leq (n+a_k) \cdot (m+b_k-1) \mbox{ \quad $x$ is of type (I)}, \\
      r_x \leq (n+a_k) \cdot b_k + (-m) \cdot a_k \mbox{ \quad $x$ is of type (II)}.
\end{align*}
\end{thm}
\begin{proof}
      We define $A$ to be $(n+a_k) (m+b_k-1)$ in Case (I)
      (to be $(n+a_k) \cdot b_k + (-m) \cdot a_k$ in Case (II)).
      We prove the inequality of this theorem by induction on $A$.
      It is easy to see that the inequality of this theorem holds for $A=0$.
      There exist $s$ and $i_t$ ($0 \leq t \leq s$) such that $ess(pg(f)) = (n+a_{i_t},m+b_{i_t})_{0 \leq t \leq s}$, where we put $a_0=b_0=0$.
      We only treat the Case (II). The proof for the case (I) is similar and omit this.
\begin{align*}
       r_x & \leq A_{ess(pg(f))} + \displaystyle \sum_{t=0}^{s-1} 2 \times S((0,0)-(n+a_{i_t},m+b_{i_t})-(n+a_{i_{t+1}},m+b_{i_{t+1}})) \\
           &= 2 \times S((0,0)-(n,m)-(n+a_k,m)-(n+a_k,m+b_k)) \\
            & \leq (n+a_k) \cdot b_k + (-m) \cdot a_k.
\end{align*}
      We use the induction assumption to prove the first inequality (see Section 6 in [\ref{ch:Oi2}]).
\end{proof}

\section{An application}
      Let $\mathfrak{F}_\chi$ be the \'etale sheaf on $U$ corresponding to $\chi$.
      We denote $\mathrm{K}_X^{\log}$ to be $\wedge^2 \Omega_X^1(\log D)$.
      We define $Sw(\chi)$ by
\begin{equation}
      Sw(\chi ) = \displaystyle \sum_{D'} Sw_{D'}(f) \cdot D', 
\end{equation}
      where $D'$ runs through all irreducible components of $D$.
      From Kato's theory in [\ref{ch:ST1}], we have
\begin{equation*}
\chi_c( U, \mathfrak{F}_\chi ) - \chi_c( U, \Q_l) = (Sw(\chi),Sw(\chi)+ \mathrm{K}_X^{\log}) - \displaystyle \sum_{x \in X} r_x.
\end{equation*} 
      Here, $\chi_c$ implies the compact support \'etale cohomological Euler-Poincar\'e characteristic,
      and $( , )$ implies the intersection pairing.
      
      From this formula and Theorem \ref{ch:simple},
      we can calculate the Euler-Poincar\'e characteristic of $\mathfrak{F}_\chi$ when $r_x=r_x'$.
      Remark that $r_x=r_x'$ for good Artin-Schreier extension (see Theorem 4.2 in [\ref{ch:Oi2}]).

\end{document}